\documentclass[12pt, a4paper]{article}
\usepackage{custom_article}

\author{Martin Callies and Andriy Haydys\\
}
\title{Affine special K\"{a}hler structures in real dimension two}
\date{14th October 2017}

\begin{document}
\maketitle

We review properties of affine special K\"ahler structures focusing on singularities of such structures in  the simplest case of real dimension two.
We describe all possible isolated singularities and compute the monodromy of the flat symplectic connection, which is a part of a special K\"ahler structure,  near a singularity.  
Beside numerous local examples, we construct continuous families of special K\"ahler structures with isolated singularities on the projective line.

\section{Introduction}

The notion of a special K\"ahler structure has its roots in physics~\cite{C11-Gates84_SuperspaceFormulation_SpKaehler, C11-deWitVanProyen84PotentialsSymmetries} and appears in supersymmetric field theories, $\sigma$-models, and supergravity. 
The following definition of an \emph{affine} special K\"ahler structure is due to Freed~\cite{C11-Freed99_SpecialKaehler}.

\begin{definition}
Let $(M,g,I,\omega)$ be a K\"ahler manifold, where $g$ is a Riemannian metric, $I$ is a complex structure, and $\omega(\cdot,\cdot)=g(I\cdot,\cdot)$ is the correspondning symplectic form.
An \emph{affine special K\"ahler structure}\index{metric!special K\"ahler} on $M$ is a flat symplectic torsion-free connection $\nabla$ on the tangent bundle $TM$ such that 
\begin{equation}
  \label{C11-Eq_SpKaehlerCondition}
(\nabla^{}_XI)Y=\, (\nabla^{}_Y I)X
\end{equation}
holds for all vector fields $X$ and $Y$.
\end{definition}

\begin{remark}
	In this contribution, we only consider affine special K\"ahler structures, which are simply abbreviated in what follows as special K\"ahler ones.  
	These should not be confused with projective special K\"ahler structures, which are not considered in this article. 
\end{remark}

There are many reasons for a mathematician to care about special K\"ahler structures. 
Perhaps one of the most important is the so called \emph{c-map construction}~\cite{C11-CFG89_GeomOfTypeIISuperstrings,C11-Freed99_SpecialKaehler, C11-MaciaSwann15_TwistGeomCmap}, which associates to a special K\"ahler structure on $M$ a hyperK\"ahler metric on the total space of $T^*M$.
Moreover, each cotangent space is a complex Lagrangian submanifold of $T^*M$ with respect to a natural complex symplectic form on $T^*M$, i.e., $\pi\colon T^*M\to M$ is a holomorphic Lagrangian fibration.
Conversely, if $\pi\colon X\to M$ is \emph{an algebraic integrable system}~\cite[Def.\,3.1]{C11-Freed99_SpecialKaehler}, then the base $M$ carries a natural special K\"ahler structure.

Another reason to care about special K\"ahler structures is that some important moduli spaces are equipped with (or are closely related to) special K\"ahler structures.  
For example, moduli spaces of Calabi--Yau threefolds are bases of algebraic integrable systems~\cite{C11-DonagiMarkman96_CubicsIntegrSystems, C11-Cortes98_OnHKmflds} and thus carry special K\"ahler structures. 
The deformation space of a compact, complex Lagrangian submanifold of a complex K\"ahler symplectic manifold is special K\"ahler~\cite{C11-Hitchin99_ModuliSpaceCxLagr}.
The Hitchin moduli space~\cite{C11-Hitchin:87} associated with a Riemann surface $\Sigma$ is the total space of an algebraic integrable system, whose base is the space of quadratic differentials on $\Sigma$. 
Hence, the space of quadratic differentials on $\Sigma$ is equipped with a natural special K\"ahler structure.


A lot is known about properties of special K\"ahler structures (especially local ones) as well as about relations of special K\"ahler structures to other geometric structures. 
For example, an extrinsic characterisation of special K\"ahler structures has been obtained in~\cite{C11-Cortes98_OnHKmflds} (see also~\cite{C11-ACD02_SpecialComplexManifolds,C11-Cortes02_SpecialMflds_Survey} as well as Section~\ref{C11-Subsect_Extrinsic} below).  
It was shown in~\cite{C11-BauesCortes01_SpKParabolicSpheres}  that a simply connected special K\"ahler manifold can be realised as a parabolic affine hypersphere.  
More recently, it has been realised~\cite{C11-AlekseevskyEtAl15_QKmetrics,C11-MaciaSwann15_TwistGeomCmap} that the c-map construction combined with the quaternionic flip of~\cite{C11-Haydys_hk:08}  is a useful tool in studies of quaternionic K\"ahler metrics (this is in turn related to projective special K\"ahler geometry).

The Riemannian geometry of the Hitchin moduli space is now being actively studied~\cite{C11-GaiottoMooreNeitzke13_WallCrossingsHitchin, C11-MSWW16_EndsOfModuliSpace}. 
The corresponding special K\"ahler structure plays a central role in the asymptotic description of the Riemannian metric near the ends of the moduli space.    


Soon after special K\"ahler structures entered the mathematical scene, Lu~\cite{C11-Lu99_NoteSpecialKaehler} proved that there are no complete special K\"ahler metrics besides flat ones.
This motivates studying singular special K\"ahler metrics as the natural structure on bases of algebraic integrable systems with singular fibers. 

The focus of this article is on singularities of special K\"ahler metrics in the simplest case of real dimension two. 
All possible singularities of special K\"ahler metrics in two dimensions (under a mild assumption) were described in~\cite{C11-Haydys15_IsolSing_CMP}.
In this introductory article, after reviewing the basics of special K\"ahler geometry, we extend the results of ~\cite{C11-Haydys15_IsolSing_CMP} by computing the monodromy of the flat symplectic connection near an isolated singularity, see Theorem~\ref{C11-Thm_Monodromy}.
We also construct continuous families of special K\"ahler structures with isolated singularities on $\P^1$, thus showing in particular that there is a non-trivial moduli space of singular special K\"ahler structures.

An interesting question, which is outside the scope of this article, is whether the result of the c-map construction applied to a singular special K\"ahler structure can be modified to yield a smooth hyperK\"ahler metric. 
A proposal for such modification was given in~\cite{C11-GaiottoNeitzkeMoore10_FourDimWallCross, C11-Neitzke14_NotesNewConstrHK}.
We leave this question for future research.

\section{Special K\"ahler geometry in local coordinates}

\subsection{Local description in terms of special holomorphic coordinates}

Locally, a special K\"ahler structure can be conveniently described in terms of special holomorphic coordinates.
Following~\cite{C11-Freed99_SpecialKaehler},  we say that a system of holomorphic coordinates $(z^{}_1,\dots, z^{}_n)$ is \emph{special}\index{coordinate system!special holomorphic}, if $\nabla (\Real \dd z^{}_j)=0$ for all $j=\overline{1,n}$.
Two special coordinate systems $\{ z^{}_j \}$ and $\{ w^{}_k\}$ are said to be conjugate\index{coordinate system!conjugate}, if $\bigl \{ p^{}_j:=\Real z^{}_j,\  q^{}_k:=-\Real w^{}_k \bigr \}$ is a Darboux coordinate system, i.e., 
\[
\om\, =\, \sum_j \dd p^{}_j\wedge \dd q^{}_j.
\]
Such coordinate systems always exist in a neighbourhood of any point~\cite{C11-Freed99_SpecialKaehler}.
Moreover, for any $j,k$ we have 
\[
\frac{\partial w^{}_j}{\partial z^{}_k}\, =\, \frac{\partial w^{}_k}{\partial z^{}_j}
\]
and therefore there is a holomorphic function $\fF$ such that 
\[
w^{}_k\, =\, \frac{\partial\fF}{\partial z^{}_k}.
\] 
This function $\fF$, defined up to a constant, is called a \emph{holomorphic prepotential}\index{holomorphic prepotential}.
The K\"ahler form can be expressed in terms of the holomorphic prepotential as follows:
\begin{equation}
  \label{C11-Eq_KaehlerFormViaPrepot}
\om \, =\, \frac \ii 2\sum\limits_{j,k}\Imag \Bigl ( \frac{\partial^2\fF }{\partial z^{}_j\partial z^{}_k} \Bigr)\;\dd z^{}_j\wedge \dd\bar z^{}_k.
\end{equation}

One more useful object, which can be attached to a special K\"ahler structure, is the so called holomorphic cubic form\index{holomorphic cubic form}, which is defined as follows.
Consider the fiberwise projection $\pi^{(1,0)}$ onto the $T^{1,0}M\subset T_\CC M$ as a 1-form with values in $T_{\mathbb C}M$. Since this form vanishes on vectors of type $(0,1)$, we can think of $\pi^{(1,0)}$  as an element of $\Om^{1,0}(M; T_\CC M)$.
Then, the holomorphic cubic form is
\[
\Xi\, :=\, -\om \bigl (\pi^{(1,0)}, \nabla\pi^{(1,0)} \bigr )\in H^0\bigl (M; \mathrm{Sym}^3 T^*M \bigr ).
\]  
In terms of the holomorphic prepotential, the holomorphic cubic form can be expressed as follows:
\[
\Xi\, =\, \frac 14\sum\limits_{j,k,l}\frac{\partial^3\fF}{\partial z^{}_j\partial z^{}_k\partial z^{}_l}\, \dd z^{}_j\otimes \dd z^{}_k\otimes \dd z^{}_l.
\]
One can show that $\Xi$ measures the difference between the flat connection $\nabla$, which is part of the special K\"ahler structure, and the Levi--Civita connection~\cite{C11-Freed99_SpecialKaehler}.

\subsection{An extrinsic description}
  \label{C11-Subsect_Extrinsic}

In \cite{C11-ACD02_SpecialComplexManifolds}, the local description of special K\"ahler manifolds in terms of special holomorphic coordinates was reformulated as an extrinsic description of simply connected special K\"ahler manifolds: 

Given an $n$-dimensional special K\"ahler manifold, then locally, special conjugate coordinate systems $\{z_j\}$ and $\{w_j\}$ define an immersion 
\[
\phi\, =\, (z^{}_1,\ldots,z^{}_n,w^{}_1,\ldots,w^{}_n)
\] 
into $T^*\CC^n = \CC^{2n}$. Thinking of $(z^{}_1,\ldots, z^{}_n,w^{}_1,\ldots w^{}_n)$ as a canonical coordinate system on $T^*\CC^n$, the standard complex symplectic form on $T^*\CC^n$ is 
\[
\Omega\, = \sum_j \dd z^{}_j\wedge \dd w^{}_j.
\] 
The immersion $\phi$ is holomorphic and Lagrangian ($\phi^*\Omega = 0$). 

Furthermore, consider the real structure $\tau\colon\CC^{2n}\to\CC^{2n}$ and $\gamma = \ii\Omega(-,\tau -)$. Then the special K\"ahler metric is given by $g = \Real(\phi^*\gamma)$ and the K\"ahler form is $\omega = \phi^*\alpha$, where $\alpha = 2\sum_j \dd p^{}_j\wedge \dd q^{}_j$. 

In particular, any simply connected special  K\"ahler manifold $(M,I,\omega,g,\nabla)$ of dimension $ n$ admits such a holomorphic, non-degenerate (i.e., $\phi^*\gamma$ is non-degenerate) Lagrangian immersion\index{Lagrangian immersion} \cite[Thm.~4(iii)]{C11-ACD02_SpecialComplexManifolds}.

The converse also holds: Let \[\phi\colon M\to T^*\CC^n = \CC^{2n}\] be a non-degenerate holomorphic Lagrangian embedding of a $n$-dimensional complex manifold $(M,I)$ and let $g=\Real(\phi^*\gamma)$ be a K\"ahler metric. Then one can prove that $\Real(\phi)$ is also an immersion and obtain global coordinates $(p^{}_1,\ldots, p^{}_n,q^{}_1,\ldots q^{}_n)$, which induce a flat torsion-free connection on $M$. With this connection, $(M, g=\Real(\phi^*\gamma), I, \omega = \phi^*\alpha, \nabla)$ is special K\"ahler (cf. \cite[Thm.~3]{C11-ACD02_SpecialComplexManifolds}).

\begin{example}\label{C11-Ex_Extrinsic}
The basic example in this context is given by a closed holomorphic $1$-form $\vartheta = \sum_j F_j dz_j$ on $M=\CC^n$ with invertible real matrix $\Imag(\tfrac{\partial F_j}{\partial z_k})$. Then, the image of 
\[
M\overset{\vartheta}{\hookrightarrow}T^*\CC^n = \CC^{2n}
\] 
is a special K\"ahler manifold. Locally, $\vartheta = \dd\fF$ is the differential of a holomorphic function $\fF$, the holomorphic prepotential. 
\end{example}

\subsection{Local description in terms of solutions of the Kazdan--Warner equation}\label{C11-Subsect_SpKandKW}

In the simplest case of complex dimension one, the following alternative local description was obtained in~\cite{C11-Haydys15_IsolSing_CMP}.
The main advantage of this description is that it does not rely on the existence of a special holomorphic coordinate. 
This is particularly important in the case of special K\"ahler structures with singularities, since near the singularities there may be no special holomorphic coordinates which extend over the singularities, and the traditional approach becomes less helpful.

\medskip

Let $\Om\subset \CC$ be any domain, which is viewed as being equipped with a holomorphic coordinate $z=x+y\ii$ and the flat Euclidean metric $|\dd z|^2=\dd x^2+ \dd y^2$.
The coordinate $z$ does not need to be special.

Write a special K\"ahler metric $g$ on $\Om$ in the form 
\[
g\, =\, e^{-u}|\dd z|^2.
\]
Using the global trivialisation of $T\Om$ provided by the real coordinates $(x,y)$ the connection $\nabla$ is described by its connection $1$-form $\om^{}_\nabla\in\Om^1\bigl(\Om;\gl(2,\RR)\bigr)$. 
A computation shows that $\nabla$ is torsion-free and satisfies~\eqref{C11-Eq_SpKaehlerCondition} if and only if $\om^{}_\nabla$ can be written in the form
\begin{equation}
  \label{C11-Eq_Conn1Form}
\om^{}_\nabla\, = \, 
 \begin{pmatrix}
 \om^{}_{11} & -*\om^{}_{11}\\
 *\om^{}_{22} & \phantom{-* }\om^{}_{22}
 \end{pmatrix}.
\end{equation}
Here $*$ denotes the Hodge star operator with respect to the flat metric.
Moreover, $\nabla$ preserves the symplectic form $\om=2e^{-u}dx\wedge dy$ if and only if $\tr\om^{}_\nabla=\om^{}_{11} +\om^{}_{22}=-\dd u$.
Thus, $\nabla$ is parameterised by a single 1-form, say $\om^{}_{11}$.

Furthermore, by a direct computation one obtains that the flatness of $\nabla$ implies that $\eta:=e^{-u}\om_{11}$ is closed.
Hence, $\nabla$ is in fact parameterized by a single \emph{closed} 1-form $\eta$; 
Moreover, $\nabla$ is flat if and only if
\begin{equation}
 \label{C11-Eq_AuxSpKaehEta}
  \begin{aligned}
  *d*\eta\, & =\, 2*(*\eta\wedge \dd u)- 2e^u |\eta|^2,\\
  \Delta u\, & =\, |2\eta + e^{-u} \dd u|^2 e^{2u}.
  \end{aligned}
\end{equation}
Here $\Delta =\partial^2_{xx} +\partial^2_{yy}$ is the Laplacian with respect to the flat metric.

Assume that any class in $H^1(\Om;\RR)$ can be represented by a harmonic 1-form. 
Then, choosing a harmonic representative $\psi$ of $[\eta]$, we can write $\eta=\psi +\tfrac 12 \dd (h + e^{-u})$.
A computation shows that~\eqref{C11-Eq_AuxSpKaehEta} is equivalent to
\begin{equation}
  \label{C11-Eq_SpKaehlerPsi}
 \Delta h=0,\qquad \Delta u= | \dd h + 2\psi |^2 e^{2u}.
\end{equation}
Hence, we obtain the following proposition, which is a slight generalisation of~\cite[Cor. 2.3]{C11-Haydys15_IsolSing_CMP}.

\begin{proposition}
For any solution\/ $(h, u,\psi)\in C^\infty(\Om)\times C^\infty(\Om)\times\Om^1(\Om)$ of
\begin{equation}
  \label{C11-Eq_SpKaehlerMostGeneral}
\Delta h\, =\, 0,\qquad \Delta\psi\, =\, 0,\qquad \Delta u\, =\, | \dd h + 2\psi |^2 e^{2u},
\end{equation}
the metric\/ $g=e^{-u}|\dd z|^2$ is special K\"ahler. 
Moreover, the connection $1$-form of the flat connection $\nabla$ is given by\/~\eqref{C11-Eq_Conn1Form} with
\begin{equation}
  \label{C11-Eq_CoeffConnectForm}
2\om^{}_{11}=  e^u(\dd h +\psi)-\dd u,\qquad
2\om^{}_{22}=  -e^u(\dd h +\psi)-\dd u.
\end{equation} 

Conversely, if any de Rham cohomology class in\/ $H^1(\Om;\RR)$ can be represented by a harmonic $1$-form, then any special K\"ahler structure on\/ $\Om$ yields a solution of Eq.~\eqref{C11-Eq_SpKaehlerMostGeneral}.\qed
\end{proposition}

In particular, for the punctured disc  $B_1^*:=B_1\setminus\{ 0 \}$ the first de Rham cohomology group is generated by the following harmonic $1$-form:
\[
\varphi= \frac{y\, \dd x - x\, \dd y}{x^2 + y^2}.
\] 
Hence, we have the following result.

\begin{corollary}[{\cite[Cor.2.3]{C11-Haydys15_IsolSing_CMP}}]
\label{C11-Cor_SpKonPunctDisc}
Any triple\/ $(h,u,a)\in C^\infty(B_1^*)\times C^\infty(B_1^*)\times\RR$ satisfying
\begin{equation}
  \label{C11-Eq_SpKaehlerDisc}
\Delta h\, =\, 0,\qquad \Delta u\, =\, |\dd h + a\varphi|^2 e^{2u}
\end{equation}
determines a special K\"ahler structure on\/ $B_1^*$. 
Conversely, any special K\"ahler structure on\/ $B_1^*$ determines a solution of~\eqref{C11-Eq_SpKaehlerDisc}.\qed  
\end{corollary}

\begin{remark}
	The last equation of~\eqref{C11-Eq_SpKaehlerMostGeneral} is the celebrated Kazdan--Warner equation\index{Kazdan--Warner equation}~\cite{C11-KazdanWarner74}.
\end{remark}

A straightforward computation shows that in the setting of Corollary~\ref{C11-Cor_SpKonPunctDisc} the holomorphic cubic form is given by
\[
\Xi\, =\, \Xi^{}_0\, \dd z^3\, =\, \frac 12\Bigl ( \frac a{2z} -\frac {\partial h}{\partial z}\ii\Bigr)\; \dd z^3.
\]
Let $N\in \ZZ$ denote the order of $\Xi$ at the origin. 
This means the following: If $N>0$, then the origin is a zero of $\Xi_0$ of multiplicity $N$; If $N<0$, then the origin is the pole of $\Xi_0$ of order $|N|$; lastly, if $N=0$, $\Xi_0$ is holomorphic at the origin and does not vanish there.
In particular, by saying that $N$ is an integer, we exclude essential singularities as well as the case $\Xi\equiv 0$, which corresponds to the flat special K\"ahler metric with $\nabla$ being the Levi--Civita connection~\cite{C11-Freed99_SpecialKaehler}.


The description of special K\"ahler metrics given in Corollary~\ref{C11-Cor_SpKonPunctDisc} can be used to prove the following result.

\begin{theorem}[{\cite[Thm.~1.1]{C11-Haydys15_IsolSing_CMP}}]
\label{C11-Thm_SingOfSpecialKaehler}
Let\/ $g=w|\dd z|^2$ be a special K\"ahler metric on\/ $B_1^*$. 
Assume that\/ $\Xi$ is holomorphic on the punctured disc and the order of\/ $\Xi$ at the origin is $N\in\ZZ$.  
Then, 
\begin{equation}\label{C11-Eq_SingOfsKmetrics}
w\, =\, -|z|^{N+1}\log |z| \bigl (C+o(1)\bigr)\qquad\text{or}\qquad w\, =\, |z|^\beta\bigl(C+o(1)\bigr )
\end{equation}
as $z\to 0$, where $C>0$ and $\beta<N+1$.

Moreover, for any $N\in\ZZ$ and $\beta\in\RR$ such that $\beta<N+1$, there is an affine special K\"ahler metric satisfying~\eqref{C11-Eq_SingOfsKmetrics}. 
(In particular, for any $N\in\mathbb Z$ there is an affine special K\"ahler metric satisfying $w=-|z|^{N+1}\log |z| (C+o(1))$.)\qed
\end{theorem}
\begin{remark}
In~\cite{C11-Haydys15_IsolSing_CMP}, the first formula of~\eqref{C11-Eq_SingOfsKmetrics} appears in the form $w=-|z|^{N+1}\log |z|\, e^{O(1)}$, which follows from McOwen's analysis of solutions of the Kazdan--Warner equation~\cite{C11-McOwen93_PrescribedCurvature}. 
The asymptotics as stated in Theorem~\ref{C11-Thm_SingOfSpecialKaehler} can be obtained from~\cite[Prop.\,3.1]{C11-MazzeoEtAl15_RicciFlow}, which in fact provides even more refined asymptotics near the origin. 
\end{remark}

By analysing the asymptotic behaviour of solutions of the Kazdan--Warner equation with singular coefficients it is possible to compute the monodromy of the flat symplectic connection. 
Namely, we have the following result, whose proof will appear elsewhere.

\begin{theorem}
	 \label{C11-Thm_Monodromy}
Let\/ $g=w|\dd z|^2$ be a special K\"ahler metric on $B_1^*$ such that
\[
w\, =\, |z|^\beta\bigl (C+o(1)\bigr ) \qquad \text{or}\qquad w\, =\, -|z|^{N+1}\log |z|\bigl (C+o(1)\bigr ),
\]
where\/ $\beta<N+1$ (in the second case, we put by definition $\beta=N+1$).
Let\/ $\Hol(\nabla)$ denote the monodromy of\/ $\nabla$ along a loop that goes once around the origin. 
Then, the following holds:
\begin{itemize} 
\item If $\beta\notin\ZZ$, $\Hol(\nabla)$ 
is conjugate to 
$\left(
\begin{smallmatrix}
\cos\pi\beta & -\sin\pi\beta\\
\sin\pi\beta & \cos\pi\beta
\end{smallmatrix}\right)$; \\
\item 
If $\beta\in 2\ZZ$, $\Hol(\nabla)$ is trivial or conjugate to $\left(\begin{smallmatrix}
1 & 1\\
0 & 1
\end{smallmatrix}\right)$; \\
\item 
If $\beta\in 2\ZZ+1$, $\Hol(\nabla)$ is $-\mathrm{id}$ or conjugate to $\left(\begin{smallmatrix}
-1 & 1\\
0 & -1
\end{smallmatrix}\right)$.\\*[-2mm]
\hfill \qed
\end{itemize}
\end{theorem}

\begin{corollary}
 \label{C11-Cor_IntegralHolonomy}
$\Hol(\nabla)$ is conjugate to a matrix lying in\/ $\Sp(2,\ZZ)$ if and only if $\beta\in \frac 12\ZZ\cup\frac 13\ZZ$.
\end{corollary}
\begin{proof}
Since $\Hol(\nabla)\in \Sp(2, \RR)=\SL(2,\RR)$, the characteristic polynomial of $\Hol(\nabla)$ has integer coefficients if and only if $\tr\Hol(\nabla)\in\ZZ$.
This implies that $\Hol(\nabla)$ is conjugate to a matrix lying in $\Sp(2,\ZZ)$ if and only if $\cos\pi\beta\in \{0, \pm\frac 12, \pm 1\}$. 
\end{proof}

\subsection{A link between two local descriptions}

Our next goal is to obtain a link between the two descriptions of special K\"ahler structures in terms of solutions of~\eqref{C11-Eq_SpKaehlerDisc} and in terms of special holomorphic coordinates.
Notice that special holomorphic coordinates always exist in a neighbourhood of a point, where the special K\"ahler structure is regular.
However, in a neighbourhood of a singular point there may be no special holomorphic coordinates.
More precisely, we have the following.

\begin{proposition}
  \label{C11-Prop_SpKFromHarmFunct}
Let\/ $\Om$ be a disc or a punctured disc. 
A special K\"ahler structure on\/ $\Om$ admits a special holomorphic coordinate on\/ $\Om$ if and only if the triple\/ $(h, u, a)$ appearing in Corollary~\textnormal{\ref{C11-Cor_SpKonPunctDisc}} is given by\/ $(h,u,a)=(h,-\log (-h), 0)$ for some negative harmonic function\/ $h$ on\/ $\Om$. 
\end{proposition}
\begin{proof}
Observe first, that for any negative harmonic function $h$ the triple $(h,u,a)=(h,-\log (-h), 0)$ is a solution of~\eqref{C11-Eq_SpKaehlerDisc}.
Moreover, in this case by~\eqref{C11-Eq_CoeffConnectForm} we have $\om_{11}=0$, which implies that $\nabla \dd x=0$.
In other words, $z$ is a special holomorphic coordinate. 

Assume now that $z$ is a special holomorphic coordinate. 
Then $\nabla \Real \dd z=0$ implies $\om_{11}=0$, which yields $\dd h + a\varphi -e^{-u}\dd u=0$. 
Since $\varphi$ is not exact, we must have $a=0$.
This yields $h=-e^{-u}$, i.e., $h$ is a negative harmonic function. 
\end{proof}

For the special K\"ahler structure determined by a single negative harmonic function as in the above proposition, we compute
\[
g\, =\, -h|\dd z|^2,\qquad \Xi\, =\, -\frac 12\, \frac {\partial h}{\partial z}\, \ii \dd z^3,\qquad
\om^{}_\nabla\, =\, \frac 1h
\begin{pmatrix}
0 & 0\\
*\dd h & \dd h
\end{pmatrix}.
\]
Furthermore, by~\eqref{C11-Eq_KaehlerFormViaPrepot} the holomorphic prepotential satisfies
\[
\Imag\frac {\partial^2\fF}{\partial z^2}\, =\, -2h.
\]
If $\Om$ is a disc, this equality determines $\fF$ up to a polynomial of degree $2$, cf.~\cite[Prop.\,1.38(c)]{C11-Freed99_SpecialKaehler}.

If $\Om=B_1^*$,  by~\cite[Thm.\,3.9]{C11-Axler01_HarmonicFuncTheory} there is $A\ge 0$ such that $h = A\log |z| + h_0$, where $h_0$ is a smooth harmonic function on $B_1$.
Hence, we have the following result.

\begin{corollary}
If a special K\"ahler structure on\/ $B_1^*$ admits a special holomorphic coordinate in a neighbourhood of the origin, there are some constants\/ $A\ge 0$ and\/ $B$ such that
\[
g=\bigl ( -A\log |z| + B +o(1) \bigr)\, |\dd z|^2\qquad\text{as} \ z\to 0,
\]
Moreover, $\ord^{}_0\Xi\ge -1$.\qed
\end{corollary}

In particular, a special K\"ahler structure on $B_1^*$ such that $\ord_0\Xi\le -2$ does not admit a special holomorphic coordinate in a neighbourhood of the origin.

\subsection{A relation with metrics of constant negative curvature}

Recall that if $g$ and $\tilde g=e^{2u}g$ are two metrics on a two-manifold, then their curvatures $K$ and $\tilde K$ are related by $\Delta u = K-\tilde Ke^{2u}$. 
In particular, if $g$ is flat, then  $\Delta u= -\tilde Ke^{2u}$.
Hence, with the help of~\eqref{C11-Eq_AuxSpKaehEta} we conclude: If $g=e^{-u}|\dd z|^2$ is special K\"ahler, then $\tilde g=e^{2u}|\dd z|^2$ has a non-positive curvature.

On the other hand, if $\tilde g=e^{2u}|\dd z|^2$ is a metric of \emph{constant} negative curvature\index{metric!of constant negative curvature} on some domain $\Omega\subset\CC$, then $u$ solves $\Delta u=Ke^{2u}$ with $K>0$.
Hence, the triple $(h, u,\psi)$ with $h(x,y)=\sqrt Kx$ and $\psi=0$ solves~\eqref{C11-Eq_SpKaehlerMostGeneral} and therefore the metric $g=e^{-u}|\dd z|^2$ is special K\"ahler.
Summarising, we obtain the following result.
\begin{proposition}[{\cite[Prop.\,3.1]{C11-Haydys15_IsolSing_CMP}}]
\label{C11-Prop_SpKFromConstNegCurv}
If\/ $\tilde g=w|\dd z|^2$ is a metric of constant negative curvature\/ $-K$ on a domain\/ $\Om$, then\/ $g=\frac 1{\sqrt w}|\dd z|^2$ is a special K\"ahler metric on the same domain\/ $\Om$. 
Moreover, the associated holomorphic cubic form is given by
\[
\pushQED{\qed}
\Xi\, =\, -\sqrt K \ii \dd z^3. \qedhere\popQED
\] 
\end{proposition}

\subsection{Examples}

\begin{example}
In the setting of Proposition~\ref{C11-Prop_SpKFromHarmFunct}, choose $h=A\log |z| + B$, where $A\ge 0$ and $B$ are some constants. 
We require also $B<0$ so that $h$ is negative on $B_1^*$.
For the corresponding special K\"ahler structure  we have
\[
g\, =\, -\bigl (A\log|z| + B\bigr ) \;|\dd z|^2,\qquad \Xi\, = - \frac {A\ii}{4z}\dd z^3,\qquad
\om^{}_\nabla=\frac A{h}
  \begin{pmatrix}
  0 & 0\\
  *\dd\log |z|\! &\!\! \dd\log |z| 
  \end{pmatrix}.
\] 
Moreover, $z$ is a special holomorphic coordinate. 
The dual ``coordinate'' $w$ is given by
\[
w\, =2 (B-A)\ii z +2 A\ii z\log z .
\] 
Of course, $w$ is not a coordinate in any neighbourhood of the origin if $A\neq 0$, but choosing a suitable branch of the logarithm the above expression defines a dual coordinate in a neighbourhood of any point in $B_1^*$.

Similarly, on a suitable branch of the logarithm (or by going to the universal covering of $B_1^*$), the extrinsic description of this metric is given by the holomorphic $1$-form
\[ 
\vartheta = w\dd z = 2 \big( (B-A)\ii z + A\ii z\log z\big) \dd z, 
\]
or the corresponding prepotential
\[
\fF = \ii (B-A) z^2 + \ii A \left(z^2 \log(z) -\frac{z^2}{2}\right)
\]
with $\frac{\partial \fF}{\partial z} = w$, as in Example~\ref{C11-Ex_Extrinsic}.

This special K\"ahler structure is related to the Ooguri--Vafa metric~\cite{C11-OoguriVafa96_SummingupDirInst}, see~\cite{C11-Chan10_OoguriVafa}.

The monodromy of $\nabla$ along the circle of radius $1$ centered at the origin can be computed explicitly and equals
\[
\Hol (\nabla)\, =\,
\left(\begin{smallmatrix}
1 & -\frac {2A\pi}B\\
0 & \ \; 1
\end{smallmatrix}\right ).
\]
\end{example}

\begin{example}
Apply Proposition~\ref{C11-Prop_SpKFromConstNegCurv} to the Poincar\'{e} metric\index{metric!Poincar\'{e}} on the punctured disc $\tilde g = |z|^{-2}\bigl (\log |z|\bigr )^{-2}|\dd z|^2$  to obtain that the metric
\[
g = -|z|\log |z|\, |\dd z|^2
\]
is special K\"ahler.
\end{example}

\begin{example}[Special K\"ahler metrics via meromorphic functions]
If $\Om$ is simply connected, the conformal factor of any metric of constant negative Gaussian curvature $K$ can be written~\cite{C11-Liouville53_SurEquation} in the form
\[
w\, =\, 4\frac {|f'(z)|^2}{\bigl( 1+K|f(z)|^2\bigr )^2},
\]  
where $f$ is a meromorphic function on $\Om$ with at most simple poles such that $f'(z)\neq 0$ on $\Om$. 
Hence, by Proposition~\ref{C11-Prop_SpKFromConstNegCurv} the metric
\[
g\, =\, \frac 12\frac {\bigl |  1+K|f(z)|^2  \bigr|}{|f'(z)|}|\dd z|^2
\]
is special K\"ahler for any meromorphic $f$ as above. 

For example, put $K=-1$ and $f(z)=z^n$, where $n\ge 1$. 
Then we obtain that 
\begin{equation*}
g=  \frac 1{2n}|z|^{1-n}\bigl (1-|z|^{2n}\bigr)\, |\dd z|^2
\end{equation*}
is a special K\"ahler metric on $B_1^*$.
\end{example}

\begin{example}
\label{C11-Ex_SpKMetrOnPuncturedPlane}
By a classical result of Picard~\cite{C11-Picard1893_Delequation}, for any given $n\ge 3$  pairwise distinct points $(z_1, \dots, z_n)$ in $\mathbb C$ and any $n$ real numbers $(\alpha_1, \dots, \alpha_n)$ such that $\alpha_j<1$ and $\sum\alpha_j>2$, there exists a metric of constant negative curvature $\tilde g$ on $\mathbb C\setminus\{ z_1,\dots, z_n \}$ satisfying $\tilde g=|z-z_j|^{-2\alpha_j}(c+o(1))|dz|^2$ near  $z_j$.  
Hence, the corresponding special K\"ahler metric $g$ has a conical singularity near $z_j$,
\[
g\, =\, |z-z_j|^{\alpha_j}(c+o(1)) |\dd z|^2.
\]

Explicit examples of constant negative curvature --- hence special K\"ahler --- metrics on the three times punctured complex plane can be found in~\cite{C11-KrausRothSugawa11_MetricsWithConicalSing} and references therein. 

\end{example}

\section{Some global aspects of special K\"ahler geometry on $\P^1$}

Even though the methods of Section~\ref{C11-Subsect_SpKandKW} are mainly local, some global conclusions can be also derived. 
The main objective for this section is to show that by allowing singular special K\"ahler metrics we have a lot of examples on a compact manifold and even a non-trivial moduli space.

\subsection{A constraint from the Gauss--Bonnet formula}
Let $g$ be a special K\"ahler metric on the complex projective line $\P^1$ with singularities at $\{ z_1,\dots, z_k \}$.
Assume that at each $z_j$ the metric $g$ has a conical singularity of order $\beta_j/2>-1$, i.e.,  
\[
g\, =\, |z|^{\beta_j}\bigl ( C_j+ o(1)  \bigr) |\dd z|^2,
\]
where $C_j$ is positive. 
A restriction on the cone angles of special K\"ahler metrics as above can be obtained from the Gauss--Bonnet formula~\cite[Prop.~1]{C11-Troyanov91_PrescribingCurvatureCmpt}, which in this case  reads
\[
\frac 1{2\pi}\int\limits_{\P^1} K\, =\, \chi(\P^1) + \frac 12\sum_{j=1}^k\beta_j.
\]
Here, $K$ is the curvature of $g$ and $\chi$ is the Euler characteristic. 
Since $K\ge 0$, compare~\cite[Rem.~1.35]{C11-Freed99_SpecialKaehler}, we obtain
\begin{equation}
 \label{C11-Eq_GaussBonnetRestr}
\sum_{j=1}^k\beta^{}_j\ge -2\chi(\P^1)\, =\, -4.
\end{equation}

\subsection{Families of special K\"ahler metrics on $\P^1$}
Just like in Example~\ref{C11-Ex_SpKMetrOnPuncturedPlane}, for any $k\ge 3$ points $z_1,\dots, z_k$ on $\P^1$ and any $\alpha_1,\dots,\alpha_k$ such that
\[
\alpha_j < 1\quad\text{and}\quad \sum_{j=1}^k\alpha_j>2,
\]
there is a unique metric $\tilde g$ of constant negative curvature on $\P^1$ with conical singularity at $z_j$ of order $-\alpha_j$.
Think of $\P^1$ as $\CC\cup\{ \infty \}$, where each $z_j$ is distinct from $\infty$. 
If $z$ is a holomorphic coordinate on $\CC$, we can write $\tilde g = w(z,\bar z)|\dd z|^2$ with
\[
w(z,\bar z)\, =\, |z|^{-4}(c+o(1))\qquad \text{as } z\to\infty.
\] 
Applying Proposition~\ref{C11-Prop_SpKFromConstNegCurv} we obtain a special K\"ahler metric $g$ on $\CC$ with conical singularity of order $\alpha_j/2$ at $z_j$ for each $j=\overline{1,k}$.
Moreover, 
\begin{equation}
  \label{C11-Eq_AsympAtInfty}
g\, =\, |z|^2\bigl (c+o(1)\bigr )\, |\dd z|^2,\qquad\text{as } z\to\infty.
\end{equation}
In other words, $g$ can be thought of as a special K\"ahler metric on $\P^1$ with conical singularities at $z_1,\dots, z_k$ and $z_{k+1}=\infty$ of order $\alpha_1/2,\dots, \alpha_k/2,$ and $-3$, respectively.
Summarising, we obtain the following result.
\begin{proposition}
  \label{C11-Prop_SpKonP1fromConstNegCurv}
For any\/ $k\ge 3$ points on\/ $\P^1$ and any\/ $\alpha_1,\dots,\alpha_k$ such that 
\[
\alpha_j<1\quad\text{and}\quad\sum_{j=1}^k\alpha_j <2,
\]
there is a special K\"ahler metric on\/ $\P^1$ such that 
\[
g\, =\, |z-z_j|^{\alpha_j}\bigl ( c_j+ o(1)\bigr) |\dd z|^2
\]
for all\/ $j=\overline{1,k}$.
Moreover, near\/ $\infty$, this metric satisfies Eq.~\eqref{C11-Eq_AsympAtInfty}, which corresponds to
\[
g\, =\, |\zeta|^{-6}\bigl ( c+o(1) \bigr ) |\dd\zeta|^2
\]
in a local coordinate\/ $\zeta$ near\/ $\infty$.\qed
\end{proposition}

Applying a M\"obius transformation, we can move $(z_1, z_2, z_3)$ into any given triple of points.
Hence, Proposition~\ref{C11-Prop_SpKonP1fromConstNegCurv} yields a family of special K\"ahler metrics with singularities at $k+1\ge 4$ points parameterised by $k+2(k-3)=3k-6$ real parameters.

\begin{remark}
Restriction~\eqref{C11-Eq_GaussBonnetRestr} does not apply to the special K\"ahler metrics constructed in Proposition~\ref{C11-Prop_SpKonP1fromConstNegCurv}, since such metrics always have singularities of order $-3$. 
\end{remark}

\bibliographystyle{alphanum}

\def\cprime{$'$}

\end{document}